\documentclass{amsart}
\usepackage{amsmath,amsthm,amssymb,amscd}
\usepackage{latexsym,amsfonts,amsthm}
\usepackage[english]{babel}
\usepackage[dvips]{graphicx}

\newtheorem{theorem}{ Theorem}[section]
\newtheorem{lemma}[theorem]{ Lemma}
\theoremstyle{definition}
\newtheorem{definition}[theorem]{ Definition}
\newtheorem{proposition}[theorem]{ Proposition}
\newtheorem{corollary}[theorem]{ Corollary}
 \newtheorem{example}[theorem]{ Example}
\newtheorem{remark}[theorem]{Remark}{\bf}{\rm}{\rm}
\newcommand{\cal}{\mathcal}
\def\F{\mu}

\def\L{\mathbb L}
\def\l{\lambda_w}
\def\lM{\lambda_w(M_L)}
\def\lk{{\mbox{lk}}}

\def\p{\partial}
\def\Q{\mathbb Q}

\begin{document}

\title{A simple formula  for the Casson-Walker invariant}
\author{Sergei Matveev and Michael Polyak}
\thanks{Both authors are partially
supported by the joint research grant of the Israel Ministry of Science (no 3-3577)
and Russian Foundation for Basic Research. The first author was also supported
by the integration project of Ural and Siberian branches of RAS. The second author was also supported by an ISF grant 1261/05.}

\address{Department of mathematics, Chelyabinsk State
University, Chelyabinsk, 454021, \linebreak \indent Russia}
\email{matveev\@@csu.ru}
\address{Department of mathematics, Technion, Haifa 32000, Israel}
\email{polyak\@@math.technion.ac.il}

\begin{abstract}
Gauss diagram formulas are extensively used to study Vassiliev link
invariants. Now we apply this approach to invariants of 3-manifolds, 
considering a manifold as a result $M_L$ of surgery on a framed link $L$ in $S^3$.
We study the lowest degree case -- the celebrated Casson-Walker invariant 
$\l$ of rational homology spheres.
This paper is dedicated to a detailed treatment of 2-component links; 
a general case will be considered in a forthcoming paper.
We present simple Gauss diagram formulas for $\lM$. This enables us to
understand/separate the dependence of $\lM$ on $L$ (considered as an 
unframed link) and on the framings. We also obtain skein relations for $\lM$ 
under crossing changes of $L$, and study an asymptotic behavior of $\lM$ 
when framings tend to infinity. Finally, we present results of extensive 
computer experiment on calculation of $\lM$. 
\end{abstract}

\maketitle

\section{Introduction}

One of the simplest types of formulas for Vassiliev invariants are
so-called Gauss diagram formulas. An invariant is calculated by
counting (with weights) subdiagrams of a special combinatorial
type in a given link diagram. The simplest example is counting
with signs crossings of two link components to get their doubled
linking number. This approach to finite type invariants was
successfully and extensively used for invariants of both classical
and virtual links (see e.g.\cite{GuPoVi,PoVi2}).

As a rule, techniques developed for link invariants were usually
later applied for 3-manifold invariants using the surgery
description of 3-manifolds. Indeed, configuration spaces integrals
(arising from the perturbative Chern-Simons theory) and the
Kontsevich integral were all adjusted and applied for 3-manifold
invariants. Surprisingly, the simplest of those -- the Gauss
diagram technique -- was not, until now, applied to invariants of
3-manifolds.

We start filling this gap by studying the case of the lowest
degree, namely the celebrated Casson invariant $\lambda(M_L)$ (or
rather its generalized versions -- the Casson-Walker invariant
$\lambda_w(M)$ and the Lescop invariant
$\lambda_L(M)=\frac{|H_1(M)|}{2}\lambda_w(M)$, see
~\cite{Walk,Lesc}). Note that $\lambda_w(M)$ is one of the
fundamental invariants of rational homology spheres. The
restriction of $\frac{1}{2}\lambda_w(M)$ to the class of integer
homology spheres is an integer extension of the Rokhlin invariant
~\cite{AkMc,Walk}. In the theory of finite type invariants of
3-manifolds (see e.g. ~\cite{GGP}) it is the simplest $\Q$-valued
invariant after $|H_1(M)|$. However, $\lambda_w(M)$ remains
 in general quite difficult to calculate. While it is easy
to do for a manifold $M_K$ obtained from $S^3$ by integer surgery
on a framed knot $K$, the same question for links remains quite
complicated, apart from the well-studied simple case of unimodular
algebraically split links. In particular, for 2-component links
satisfactory formulas only exist for some special
cases~\cite{KiLi}. Formulas from \cite{Lesc}, although explicit,
were not much applied or studied, possibly due to a large number
of terms of various nature and their complicated or cumbersome
definitions. While Lescop's general formulas for $\lambda_L(M)$
seem to imply all other formulas of \cite{KiLi, Joha1, Joha2},
including also formulas of this paper, we feel that a simple
explicit formula remains of a considerable interest.

Our approach is elementary, so allows a simple computation of the
invariant directly from any integer-framed link diagram and makes
many of its properties transparent. We proceed by the number of
components of a given framed link $L$. When $L$ is a knot, $\lM$
is closely related to the simplest finite type knot invariant
$v_2$ -- the second coefficient of the Alexander-Conway
polynomial. So its properties are easy to study knowing the
behavior of $v_2$; in particular, the Gauss diagram formula is
well-known (see \cite{PoVi1,PoVi2}). This paper is dedicated to a
similar detailed treatment of the 2-component case. The general
$n$-component case will be discussed in a forthcoming paper.

A special case of two component links with zero-framed components
was studied in \cite{KiLi}. While the invariant is described there
only by means of its behavior under a self-crossing of one of the
components, a Gauss diagram formula is in fact implicit there. We
show that their results extend to the case of arbitrary integer
framings and provide a simple Gauss diagram formula in this
general case. The main ingredient of this formula appeared first
in \cite{PoVi1} and is now known as the generalized Sato-Levine
invariant, see \cite{AMR,AR,KiLi,Mel}. The Gauss diagram formula
enables us to calculate the values of $\lM$ in a simple and
straightforward way starting from any diagram of a 2-component
link.

Our additional goal is to understand and separate the dependence
of $\lM$ on the underlying topological type of the link and on the
framings. This is also evident from the formula. In particular, we
observe an interesting asymptotics  of the Casson-Walker invariant
as both (or just one of) the framings are rescaled and taken to
infinity. It is interesting that in two special cases the
asymptotical behavior of $\lM$ turns out to be very different from
the general case. The first case appears when the framing of one
of the components is zero. This could be expected and reflects the
fact that the manifold obtained by surgery on this component is
not a homology sphere. Another -- somewhat more surprising --
special case is when the framings of two components are opposite
to each other.

We also deduce a simple skein relation for $\lambda_w(M)$ which
relates its change under a crossing change between two components
to its values on the smoothed link and the sublinks.

This has interesting implications for the theory of finite type (or
perturbative) invariants of 3-manifolds. Indeed, this theory is
well-studied only for integer homology spheres, and thus was defined
only on special classes of surgery links, with a complicated Borromean-type modification \cite{Mat} playing the role of a crossing change.
The behavior of an invariant under a self-crossing of one of the link
components $L_i$ easily fits into the theory (since it may be obtained
by adding a new small 1-framed component going around two strands of
$L_i$ near the crossing), and thus is well-understood \cite{Joha2}.
A crossing change between different components, however, changes the
homology of the resulting manifold, so does not fit the theory of
finite type invariants of 3-manifolds. Our work implies, however,
that the relation between finite type invariants of links and
3-manifolds is closer than one could expect and that the behavior of
these invariants under such crossing changes may be understood as well.

We remark that while in this paper we usually restrict our consideration
to the case when $M_L$ is a rational homology sphere, our formulas
are well defined also in case when $M_L$ is not such. In this case
our formula gives the Lescop's generalization of the Casson-Walker
invariant.

We sum up the goals of this paper in the following

{\bf Problem.} Given a framed 2-component link $L=\{L_1,L_2\}$
with the integer linking matrix $\L$, we want to
\begin{enumerate}
\item  Find simple Gauss diagram formulas for $\lM$;

\item Understand/separate the dependence of $\lM$ on $L$
(considered as an unframed link) and on $\L$.

\item Obtain skein relations for $\lM$ under crossing
changes of $L$.

\item Study an assymptotic behavior of $\lM$ when framings
tend to infinity.
\end{enumerate}

We thank Nikolai Saveliev for useful discussions and Vladimir
Tarkaev for writing a computer program. The main results of the
paper have been obtained during the stay of both authors at MPIM
Bonn. We thank  the institute for hospitality, creative atmosphere,
and support.

\section{Arrow diagrams}
Let $A$ be an oriented 3-valent graph whose edges are divided into
two classes: {\em fat } and {\em thin}. The union of all vertices
and all fat edges of $A$ is called a {\em skeleton} of $A$.

\begin{definition} \label{defarrowd} $A$ is called an {\em arrow
diagram}, if  the skeleton of $A$ consists of disjoint circles.
Thin edges are called {\em arrows}.
\end{definition}

It follows from the definition that each arrow connects two
vertices, which may lie in the same circle or in different
circles. By a {\em based} arrow diagram we mean an arrow  diagram
with a marked point in the interior of  one of its fat edges. See
Fig.~\ref{arrex} for simple examples of based and unbased arrow
diagrams.
\begin{figure}[ht]
\centerline{\includegraphics[height=2.2cm]{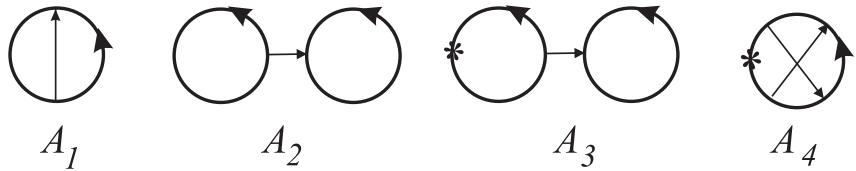}}
\caption{Simple examples of  arrow diagrams. }
  \label{arrex}
\end{figure}

\section{Gauss   diagrams}
  Recall that an $n$-component link diagram
is a generic immersion of the disjoint union of $n\geq 1$ oriented
circles to plane, equipped with the additional information on
overpasses and underpasses at double points. Any link diagram can be
presented numerically by its Gauss code, which consists of several
strings of signed integers. The strings are obtained by numbering
double points and traversing the components. Each time when we pass
a double point number $k$, we write $ k$ if we are on the upper
strand and   $-k$  if on the lower one. If we prefer to distinguish
knots and their mirror images or if we are considering a link with
$\geq 2$ components, then an additional string of $\varepsilon_i=\pm
1 $ called {\em chiral signs} is needed. Here $i$ runs over all
double points $a_i$ of the diagram and  signs are determined by the
right hand grip rule. See Fig.~\ref{exampleMFr}.

\begin{figure}[ht]
\centerline{\includegraphics[height=4cm]{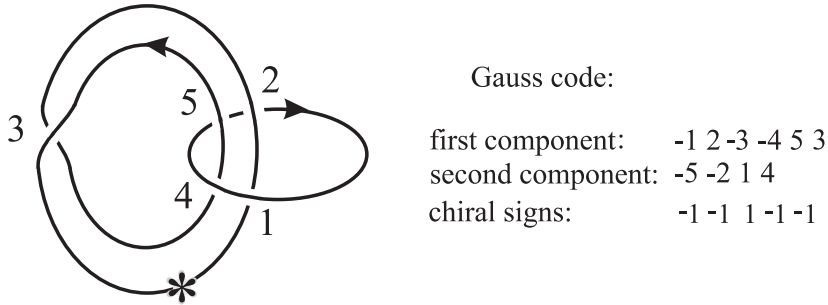}}
 \caption{A 2-component link and its Gauss code}
 \label{exampleMFr}
\end{figure}

A convenient way to visualize a Gauss code is a {\em Gauss diagram}
consisting of the oriented link components with the preimages of
each double point connected with an arrow from the upper point to
the lower one. Each arrow $c$ is equipped with the chiral sign   of
the corresponding double point. The numbering of endpoints of arrows
is not necessary anymore, see Fig.~\ref{Gaussd}. We say that a link
diagram is {\em based} if a non-double point in one of its
components is chosen. An equivalent way of saying this consists in
considering {\em long links} in $R^3$, when the base point is placed
in infinity. If the link is based, then the corresponding Gauss
diagram is based too. Note that forgetting signs converts any Gauss
diagram into an arrow diagram, but not any arrow diagram (for
example, a fat circle with two thin oriented diameters) can be
realized by a Gauss diagram of a link. However, that is possible,
if we allow virtual links. This is actually the main idea of the
virtualization.

\begin{figure}[ht]
\centerline{\includegraphics[height=2.8cm]{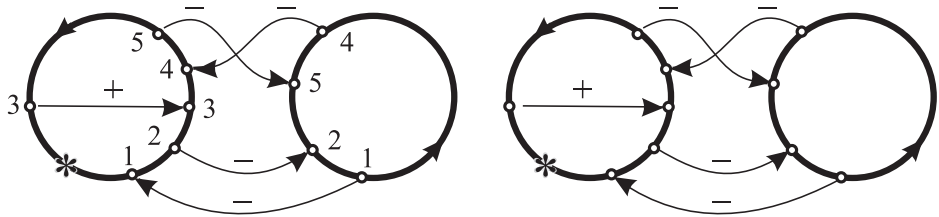}}
 \caption{ Two presentations of the  Gauss diagram  for the link
  diagram in Fig.~\ref{exampleMFr} }
 \label{Gaussd}
\end{figure}

\section{Arrow diagrams as functionals on Gauss diagrams}
As described in~\cite{PoVi1,GuPoVi}, any arrow diagram $A$ defines an
integer-valued function $\left  < A,* \right >$   on set of all
Gauss diagrams. Let $A$  be an $n$-component arrow diagram and $G$
be an $n$-component Gauss diagram. By a {\em representation} of $A$
in $G$ we mean an embedding of $A$ to $G$ which takes the circles
and arrows of $A$ to the circles, respectively,
     arrows of $G$ such that the orientations of all circles and
 arrows are preserved. If both diagrams are based, then
 representations must respect base points. For a given
 representation $\varphi \colon A\to G$
 we define its {\em sign}  by
$\varepsilon(\varphi)=\prod \varepsilon(\varphi(a))$, where the
product is taken over all arcs   $a\in A$.

\begin{definition} \label{functional} Let $A$ be an $n$-component arrow
diagram. Then for any  $n$-component arrow diagram $G$ we set
$\left <A,G\right >=\sum \varepsilon(\varphi) $, where the sum is
taken over all representations of $A$ in $G$.
\end{definition}

\begin{example} \rm  \label{lk}
Let us describe functions for arrow diagrams $A_1 -A_4$  shown in
Fig.~\ref{arrex}. Evidently, $A_1$ determines the writhe of the
link, which is defined as the sum of the chiral signs of all
double points. Let $G$ be a Gauss diagram of an oriented
2-component link $L=L_1\cup L_2$. Then $\left <A_2,G\right
>=2n$, where $n=\lk (L_1,L_2)$ is the linking number  of the
components. Indeed, for any arrow of $G$ we have exactly one
representation of $A_2$ to $G$. Therefore, all double points
contribute to $\left <A_3,G\right >$ (not only points where the
one preferred component is over the other).  If we insert a base
point into $A_2$ and a base point into $G$ (thus fixing ordering
of the two link component), we get $n$ without doubling. It
means that $\left <A_3,G\right >=\lk(L_1,L_2)$. The meaning of
$\left <A_4,G\right >$ is more complicated. If $G$ is a Gauss
diagram of a knot $K\subset S^3$, then $\left <A_4,G\right >$ is
the second coefficient $v_2$ of the Conway polynomial of $K$,
which is often called the Casson invariant of $K$. See~\cite{PoVi2}
for a diagrammatic description  and properties of $v_2$.
\end{example}

\begin{example} \rm  \label{exA1A2} For arrow diagrams $U_1, U_2$
shown in Fig.~\ref{arroval} and the Gauss diagram $G$ shown in
Fig.~\ref{Gaussd}, we have $\left <U_1,G\right
>=0$ and $\left <U_2,G\right >=1$.
 The image of the unique representation of $U_2$ to
$G$ contains arrows 2,3,4.
\end{example}
\begin{figure}[ht]
\centerline{\includegraphics[height=1.5cm]{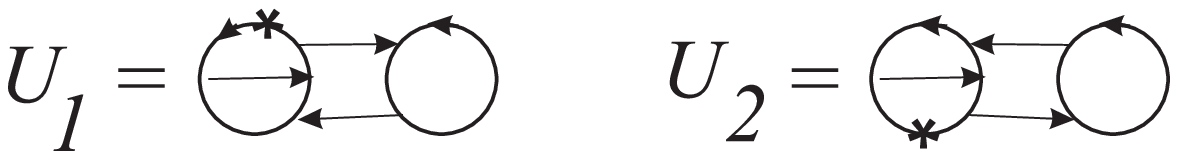}}
 \caption{
Two arrow diagrams}
 \label{arroval}
\end{figure}

\begin{remark} \rm \label{firstrem} Often it is convenient to extend
Definition~\ref{functional} by linearity to the free abelian group
generated by arrow diagrams. Let $A=\sum_{i=1}^mk_iA_i$ be a linear
combination of arrow diagrams. In general, the value $\left
<A,G\right >$ depends on the choice of the Gauss diagram $G$ of a
given link $L$ as well as on the choice of the base point. However,
for some carefully composed linear combinations of arrow diagrams
the result does not depend on the above choices. This gives a link
invariant $\left <A,G\right >$, which we will denote by $A(L)$.
It is easy to show that such an invariant is of finite type.
Moreover, any finite type invariant of long knots can be presented
in such a form. A similar result for links is unknown. See~\cite{GuPoVi}.
\end{remark}

\section{An arrow diagram invariant of degree three}

 Consider the  following linear combination $U=U_1+U_2+U_3+U_4$
of arrow diagrams (see Fig.~\ref{basicformula}).

\begin{figure}[ht]
\centerline{\includegraphics[height=1.8cm]{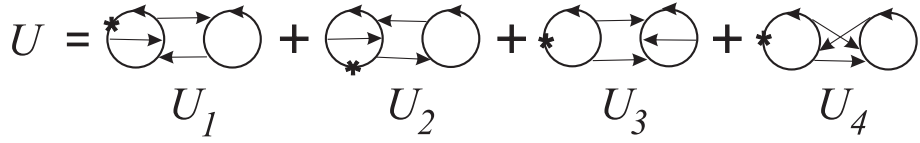}}
\caption{Remarkable linear combination of arrow diagrams}
 \label{basicformula}
\end{figure}

\begin{example} \rm For the link and Gauss diagrams shown in
 Fig.~\ref{exampleMFr} and Fig.~\ref{Gaussd} we have
 $\left <U,G\right > = 1-1=0$, since there are only two
 representations of  summands of $U$ in $G$: one representation
 of $U_2$ (see Example~\ref{exA1A2}) and one representation of
 $U_4$ with arrows 2,4,5 in its image.
\end{example}

\begin{example} \rm For the link and Gauss diagrams shown
 in Fig.~\ref{favor} we have $\left <U,G\right > = -1-1=-2$,
 since no summands of $U$ have representation in $G$ except
 the first one which has  two representations described by
 arrow triples (2,5,7) and (4,5,7).
\end{example}

\begin{figure}[ht]
\centerline{\includegraphics[height=4cm]{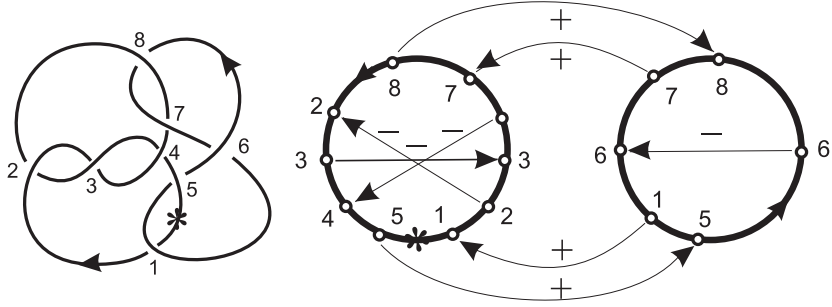}}
 \caption{Link $8_{11}^2$ in the Alexander-Briggs-Rolfsen
 table~\cite{rolf}.
  }
 \label{favor}
\end{figure}

Let $n$ be an integer and  $B$ an unknotted annular band having
$n$ negative full twists if $n\geq 0$ and $|n|$ positive full
twists if $n< 0$, see Fig.~\ref{ghl} for $n=5$. If we equip
the components of $\partial B$ by orientations induced by an
orientation of $B$, then their linking number is equal to $n$.

\begin{definition}\label{HopfLink} The 2-component link $\p B$
is called the {\em generalized Hopf link} and denoted $H(n)$.
Denote by $H(n,a,b)$ the generalized Hopf link
$H(n)$ with framings $a,b$ of its components.
\end{definition}

\begin{figure}[ht]
\centerline{\includegraphics[height=4cm]{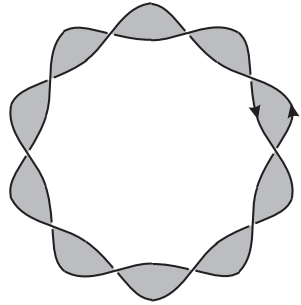}}
\caption{Generalized Hopf link $H(n)$ for $n=5$.}
 \label{ghl}
\end{figure}

The orientation of components of $\p B$ induced by an orientation
of $B$ is a part of definition of $H(n)$. If we reverse
orientation of one of them, we get a new oriented link
$\bar{H}(n)$ with the linking number of components $-n$. A diagram
of $H(n)$ and $\bar{H}(n)$ for $n=3$ are shown on top and bottom
of Fig.~\ref{ghl3}, respectively.

\begin{example} \rm \label{H(3)}
The diagrams of $H(n)$ and $\bar{H}(n)$ mentioned above differ
only by orientation of one component. Nevertheless, their Gauss
diagrams $G_n$ and $\bar{G}_n$ look quite different, see
Fig.~\ref{ghl3} for $n=3$. For $G_3$ we have $\left <U,G_3\right >
= 0$, since no summands of $U$ have representation in $G_3$. Of
course, the same fact holds for any $n$. For the bottom diagram
$\bar{G}_3$ we get $\left <U,\bar{G}_3\right >=-4$. Indeed, in
this case there are four representations of $U_4$  in $\bar{G}_3$.
They can be described by four triples of positive arrows
(1,2,3),(1,2,5), (1,4,5),(3,4,5) contained in their images. Since
$0=\left <U,G_3\right >\neq  \left <U,\bar{G}_3\right >=-4$, we
may conclude that the value of $\left <U,G\right >$ depends of
orientation of the components. Nevertheless, the following
proposition shows that aside this phenomenon $\left <U,G\right >$
is invariant.

\end{example}
\begin{figure}[ht]
\centerline{\includegraphics[height=7cm]{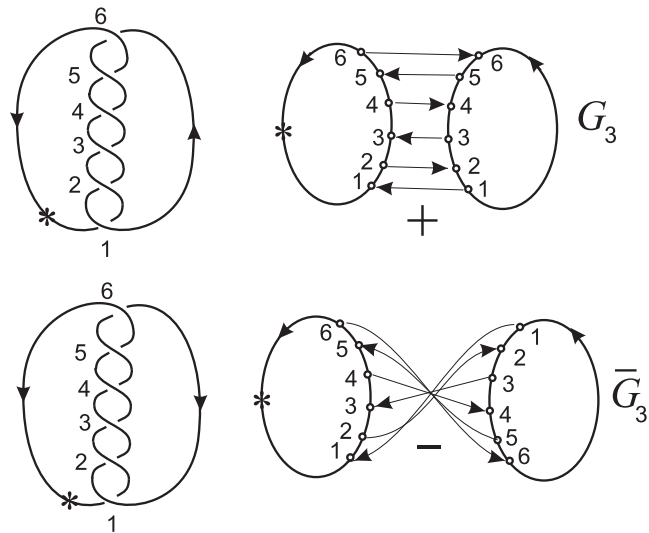}}
 \caption{Links $H(3)$ and $\bar{H}(3)$}
 \label{ghl3}
\end{figure}

\begin{proposition} \label{prop1} $U$ determines an invariant
$U(L)=\left <U,G\right >$ for oriented links of two ordered
components.
\end{proposition}

\begin{proof} We will always assume that the base point is on
the first component. It suffices to show that $\left <U,G\right >$
is invariant with respect to
 \begin{enumerate}
 \item   Reidemeister moves far from the base point.
 \item Replacing the base point within the first component.
 \end{enumerate}

In order to verify the invariance under Reidemeister moves, it
suffices to show that $\left <U,R\right >=0$, where $R$ runs over
all relations of the Polyak algebra $P$ defined in~\cite{GuPoVi}.

Invariance of $\left <U,G\right >$ under replacing the base point
within the same component follows from the observation that two
link diagrams with the common base point in their first components
are Reidemeister equivalent if and only if they are equivalent via
Reidemeister moves performed far from the common base point. One
can give a ``folklore'' reformulation of that fact by saying that
the theory of links is equivalent to the theory of {\em long links}
(having in mind that the base point is at infinity).
\end{proof}

Once the invariance of $U(L)$ is established, one can identify it
with the {\em generalized Sato-Levine invariant}
$c_3(L)-(c_2(L_1)+c_2(L_2))c_1(L)$ (see \cite{AMR,AR,KiLi,Mel}),
where $c_k$ is the coefficient at $z^k$ of the Conway polynomial.
There are several ways to do that. In particular, one may show
that $\left <U,G\right >$ does not depend on the ordering of the
components of $L$. Then the identification follows from the fact
that it is an invariant of degree three for 2-component links and
such invariants are classified. So it suffices to check values of
$U(L)$ on a few simple links. Another way to identify $U(L)$ is to
check that it satisfies a simple skein relation under crossing
change of one of the components. Then its values on generalized
Hopf links determine the invariant.

We will follow the latter scheme of proof, since it will also be
the easiest way to relate the function $\F(L)$ from Definition
\ref{formula} below and $\lM$. Alternatively, instead of checking
values of $U(L)$ on generalized Hopf links, one may check that
$U(L)$ satisfies an appropriate skein relation under crossing
change of two different components (and vanishes on the unlink).
This idea will be used in Section \ref{sec:skein} to extract a
skein formula for the Casson-Walker invariant.

\section{The Formula}

Let $ {L}=L_1\cup L_2$ be an oriented  2-component framed link.
Denote by $a,b$, and $n$ the framings of $L_1,L_2$, and their
linking number, respectively. Suppose that the integer linking
matrix $\L=\left( \begin{array}{cc}a &n \\ n &b\end{array}\right)$
of $L$ has non-zero determinant $D=\mathrm{Det}(\L)$. The
signature of $\L$ will be denoted by $\sigma$. It is easy to see
that   $\sigma $   can be found by the following rule:

\[ \sigma = \left \{ \begin{array}{ll}
  0 & \mbox{if   $D<0$}\\
  2 & \mbox{if   $D>0$ and  $a+b>0$}\\
 -2 & \mbox{if   $D>0$ and  $a+b<0$}\\
 \end{array} \right.\]
We point out that if $D>0$, then $a,b$ have the same sign.
Therefore, in this case $\sigma$ is determined only by the sign,
say, of $a$.

Recall that $v_2(K)$ denotes the Casson invariant
$\left<A_4,G_K\right>$ of a knot $K\subset S^3$ (see
Fig.~\ref{arrex} for arrow diagram $A_4$). It coincides with the
coefficient at $z^2$ of the Conway polynomial of $K$. It can also
be extracted from the Alexander polynomial $\Delta_K (t)$
normalized so that $\Delta_K(t^{-1})=\Delta_K(t)$ and
$\Delta_K(1)=1$ as follows: $v_2(K)=\frac{1}{2}\Delta_K''(1)$.

We introduce a function $\F\colon{\cal L} \to \Q$, where $\cal L$
is the set of all oriented framed 2-component links, as follows.
Let $L=L_1\cup L_2\in {\cal L}$ .

\begin{definition}\label{formula}
$$\F(L)=av_2(L_2)+bv_2(L_1)-U(L)+\frac{1}{12}(n^3-n)+\frac{1}{24}
 (a+b)(2n^2-ab -2) $$
\end{definition}

\begin{example} \rm  \label{compF} Let us calculate $\F(L)$, where $L=H(n,a,b)$ is  the
generalized Hopf link $H(n)$ framed by $a,b$, see Definition~
\ref{HopfLink}. It is easy to see that $U(H(n))=0$ (see
Example~\ref{H(3)}, where we have shown that $U=0$ for $H(3)$).
Taking into account that $v_2(L_1)=v_2(L_2)=0$, we get
$$ \F(H(n,a,b))=\frac{1}{12}(n^3-n)
+\frac{1}{24}(a+b)(2n^2-ab-2)
$$
\end{example}

Let $M_L$ be the manifold obtained by surgery on $L$. Note that
our assumption  $D\ne 0$ means that the corresponding 3-manifold
$M_L$ is a rational homology sphere, i.e. the first homology group
of $M_L$ is finite. Its order $|H_1(M_L)|$ equals to $|D|$.

Let $\lM$ be the Casson-Walker invariant of $ M_{L}$, see
~\cite{Walk}. We normalize it following Walker, so as to have
$\frac{1}{2}\lambda_w(P_{120})=1$, where $P_{120}$ is the
positively oriented Poincar\'{e} homology sphere, obtained from
$S^3$ by surgery along the  trefoil with framing 1.

\begin{theorem}\label{main}(Main) For any oriented framed
2-component link $L=L_1\cup L_2$ we have
$${\frac{1}{2}D(\lM-\frac{1}{4}\sigma)=\F(L)}$$.
\end{theorem}

\begin{example} \rm  \label{complambda} Let us apply the main
theorem for calculation of the Casson-Walker invariant $\lambda_w$
for the manifold $M_{H(2,3,1)}$, where $ H(2,3,1)$ is the Hopf
link $H(2)$  framed by $a=3,b=1$ (see   Fig.~\ref{negaho2}).
\begin{figure}[ht]
\centerline{\includegraphics[height=3cm]{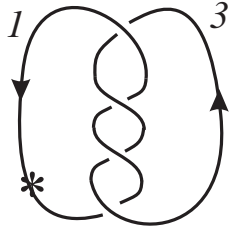}}
 \caption{Generalized  Hopf link $H(2)$ framed by $3,1$}
 \label{negaho2}
\end{figure}
It follows from Example~\ref{compF} that
$$\frac{1}{2}\l(M_{H(2,3,1)})=\frac{1}{24D}(12+(a+b)(2n^2-ab-2)
+\frac{1}{8}\sigma .$$ Substituting $D=3\cdot1-2^2=-1$ and $\sigma
=0$ we get $\frac{1}{2}\lambda_w(M_{H(2,3,1)})=-1$. This is not
surprising, since $M_{H(2,3,1)}$ is the negatively oriented
Poincar\'e homology sphere $-P_{120}$. Indeed, one generalized
destabilization move (sometimes called {\em blow-down}) transforms
$H(2,3,1)$ into the left trefoil framed by -1.
\end{example}

The plan of proving the above theorem consists of two
steps.\label{plan}

Step 1. We show that the correctness of equality
$\frac{1}{2}D(\lM-\frac{1}{4}\sigma )=\F(L)$ is preserved under
self-crossing of a link component.

Step 2. We show that the equality is true for the Hopf links
$H(n), n\neq 0$ framed by $a,b$. We do that by comparing our
formula for $\F(H(n,a,b))$ (see Example~\ref{compF}) with Lescop's
formula for the values of $\lambda_L$ for  Seifert manifolds (see
\cite{Lesc}, page 97).

Steps 1,2 imply Theorem~\ref{main}, since any 2-component link
can be transformed by self-crossings of its components into a
generalized Hopf link.

Let us pause to make few remarks on the role of the function $\F$.
\begin{remark}
Theorem \ref{main} implies that $\F(L)$ does not depend on the
orientation of the components and is preserved under handle
slides. It is not a pure 3-manifold invariant, since some
additional information from the surgery link is needed (namely,
the sign of $D$ and the signature $\sigma $). However, it may be
considered as a convenient renormalization of the Casson invariant
for technical purposes, since it is better suited for calculations
and has simpler skein properties.
\end{remark}
\begin{remark}
It turns out that $\F$ may be defined by a similar Gauss diagram
formula for $n$-component links with any $n=1,2,3,\dots$. The
general case $n\ge3$ will be discussed in a forthcoming paper. Let
us define it for $n=1$. For a framed knot $K$ with framing $a$ we
denote $$\F(K)=v_2(K)-\frac{1}{24}(a^2+2), \quad \text{where}\quad
v_2(K)=\left<A_4,G_K\right>$$ Since values of $\lambda_w$ for
manifolds obtained by surgery on a framed knot are well-known, it
is easy to check that similarly to Theorem \ref{main} one indeed
has
$$\displaystyle{\frac{1}{2}D(\lambda_w(M_K)-\frac{1}{4}\sigma(a))=\F(K),}$$
where $\sigma(a)$ is the signature of $(1\times1)$-matrix $(a)$,
i.e. the sign of $a$. This may be deduced also from Theorem
\ref{main}, by considering $K$ stabilized by $\pm1$-component. Now
we can rewrite $\F(L)$ for 2-component links using its values on
sublinks:
$$\F(L)=\F(L_1)b+\F(L_2)a-U(L)+\frac{1}{12}(n^3-n)+\frac{1}{12}(a+b)n^2$$
\end{remark}

\section{Behavior of  $\F$ with respect to self-crossings}
Let us carry out Step 1. Suppose that a diagram $G^-$ of an
oriented  framed 2-component link $L^-=K^-\cup S^-$ is obtained
from a diagram $G^+$ of a framed link $L^+=K^+\cup S^+$ by a
single crossing change at a double point $C$ of $K^+ $ such that
the chiral sign of $C$ is 1 in $K^+ $ and -1 in $K^-$.
  Note that  $K^+$ can be
considered as to consist of two loops ({\em lobes}) with endpoints
in $C$. Denote by $\ell$ the linking number of the lobes, by $k$
the linking number of one of the lobes with $S^+$, and by $n$ the
linking number of $K^+$ and $S^+$. Note that the crossing change
preserves the linking matrix $\L^+$ of $L^+$.

\begin{lemma}\label{step1}  In the situation above we have
$$\frac{1}{2}D(\lambda_w(M_{L^+})-\lambda_w
(M_{L^-}))= b\ell - k(n-k)=\F(L^+)-\F(L^-),$$
 where $D=\mathrm{Det}(\L^+)$ and $b$ is the framing of $S^+$.
\end{lemma}
\begin{proof} The first equality is a partial case (for 2-component links)
of the crossing change formula, which is the main result
of~\cite{Joha1}.

Let us prove the second one. Since the crossing change preserves
the linking matrix, we have $\F(L^+)-\F(L^-)=
bv_2(K^+)-bv_2(K^-)-\left <U,G^+\right
>+\left <U,G^-\right >$. Recall that the equality
$v_2(K^+)-v_2(K^-)= \ell$ is one of the main properties of $v_2$,
see~\cite{GuPoVi}. For the Arf-invariant $v_2$\ {\rm mod} 2 it was
known long ago, see~\cite{Mat}. It follows that
$bv_2(K^+)-bv_2(K^-)=b\ell$.

Let us compute  $\left <U,G^+\right
>-\left <U,G^-\right >$. Note that
the Gauss diagrams $G^\pm$ are almost identical. The only
difference between them is that the arrows $a^+(C), a^-(C)$
corresponding to $C$ have opposite orientations and signs. Let us
analyze representations of $U_k,1\leq k\leq 4,$  in $G^\pm$. We
call a representation $\varphi \colon U_k\to G^\pm$ {\em
significant}, if its image contains $a^\pm(C)$. Otherwise
$\varphi$ is {\em insignificant}.   Since $G^\pm$ are identical
outside $a^\pm(C)$, there is a natural bijection between
insignificant representations of $U_k$ to $G^+$ and $G^-$ such
that the signs of corresponding representations are equal. It
follows that insignificant representations do not contribute to
the difference $U(L^+)-U(L^-)$.

{\sc Case 1}. Assume that the base point of $L^+$ is in $K^+$.
Denote by $P$ the lobe of $K^+$ containing the base point. Let
$\bar P$ be the other lobe. Consider a significant representation
$\varphi  \colon U_k \to G^\pm$. Since $\varphi  $ is significant,
$a^+(C)$ is in its image and thus $k=1$ or $k=2$. Any such
representation is completely determined by two arrows of $G^\pm$
contained in the image of $\varphi$. One of them connects $P$ to
$S^\pm$, the other goes from $S^\pm$ to $\bar P$. It is easy to
see that the contribution of $\varphi $ to the difference
$U(L^+)-U(L^-)$ is equal to the product of the signs of those two
arrows.    Taking into account Example~\ref{lk} (containing Gauss
diagram description of various linking numbers) we may conclude
that $U(L^+)-U(L^-)=k(n-k)$.

{\sc Case 2}. Assume that the base point of $L^+$ is in $S^+$.
Since $C$ is in $K^+$, there are no significant representations of
$U_1,U_2,U_4$  in $G^\pm$. Moreover, any representation  $\varphi
\colon  U_3 \to G^\pm$  is completely determined by two arrows of
$G^\pm$ contained in the image of $\varphi^\pm $. One of them
connects $S^\pm$ to $P$, the other   $S^\pm$ to $\bar P$. As
above, the contribution of $\varphi $ to the difference
$U(L^+)-U(L^-)$ is equal to the product of the signs of those two
arrows. We may conclude that $U(L^+)-U(L^-)=k(n-k)$ also in this
case.
\end{proof}

\begin{example} \label{crossingex} Let $L^+=L_1^+\cup L_2^+$ be
the link shown in Fig.~\ref{lke} (all arrows have positive signs).
Suppose that $L^-=L_1^-\cup L_2^-$ is obtained from $L^+$ by a
crossing change at the double point $C$ having number 1. We assume
that the framings of $L_1^+, L_2^+$ are $a,b$. Note that $n=3$,
$k$ is either 1 or 2, and $\ell =1$. Taking into account that
$v_2(L_1^+)=1$  and $v_2(L_2^+)=0$, we get $\F(L^+)-\F(L^-)=b-2$.
Therefore,
$\frac{1}{2}D\lambda_w(M_{L^+})-\frac{1}{2}D\lambda_w(M_{L^-})=b-2$.
\end{example}
\begin{figure}[ht]
\centerline{\includegraphics[height=4cm]{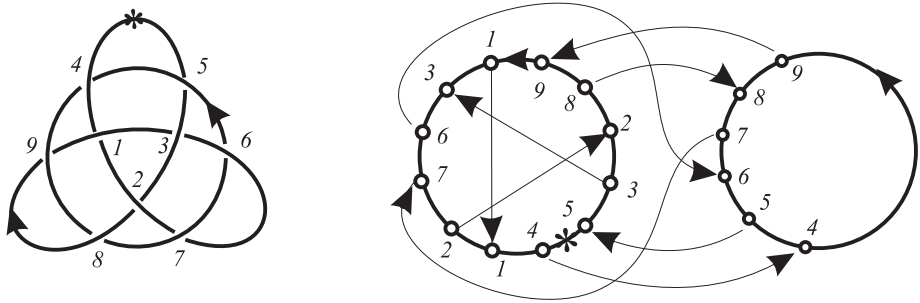}}
 \caption{There is only two representations of $U_k, 1\leq k \leq 4$ to $G^\pm$, namely
 a representation $U_1 \to G^+$ having sign 1 and a representation $U_2 \to G^-$
 having sign -1. They    covers arrows $(1,4,7)$ of $G^+$ and $(8,7,\bar 1)$ of
 $G^-$, where $\bar 1$ denotes the first arrow of $G^+$ with reversed orientation. The total contribution of $U$ to
 $\F(L^+)-\F(L^-)$ is 2}
 \label{lke}
\end{figure}

\begin{corollary} \label{basepoint1} Let $L=L_1\cup L_2$ be
an oriented 2-component link in $S^3$ and $U$ the linear
combination of arrow diagrams shown in Fig.~\ref{basicformula}.
Then the invariant $U(L)=\left <U,G\right >$ does not depend on
the ordering of the components of $L$. It coincides with the
generalized Sato-Levine invariant
$c_3(L)-(c_2(L_1)+c_2(L_2))c_1(L)$, where $c_k$ is the coefficient
at $z^k$ of the Conway polynomial.
\end{corollary}

\begin{proof} Let us compare values of $U$ for two links which
differ only by ordering. By Lemma \ref{step1} the behavior of $U$
under self-crossings does not depend on the ordering. Thus it
suffices to compare values of $U$ on Hopf links $H(n)$ with two
different orderings; but Example \ref{H(3)} shows that it is $0$
in both cases. The last statement follows from the fact that the
generalized Sato-Levine invariant satisfies the same skein
relations and also vanishes on $H(n)$.
\end{proof}

\section{Model manifolds $Q(n,a,b)$}
\label{sectionGHL}

This section is dedicated to the study of the 3-manifold obtained
by surgery on $H(n,a,b)$. To that end, let us introduce another
framed link $S(A,B,C)$ shown in Fig.~\ref{linkS}.
It consists of four unknotted circles framed by $A=a+n,B=b+n,C=-n$,
and 0 such that each of the first three circles links the forth
circle framed be 0 exactly ones.

\begin{figure}[ht]
\centerline{\includegraphics[height=4.5cm]{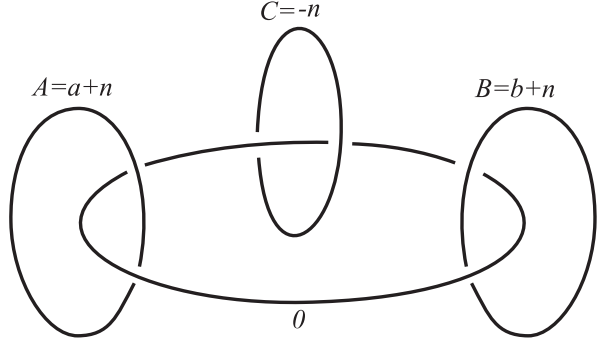}}
 \caption{Framed link presenting a Seifert manifold }
  \label{linkS}
\end{figure}

\begin{lemma}\label{Q=Q'} Manifolds obtained  by Dehn surgery of $S^3$
along  $H(n,a,b)$ and $S(A,B,C)$ are related by a  homeomorphism
which preserves orientations induced from the orientation of $S^3$.
\end{lemma}
\begin{proof} Adding the C-framed component to the components
framed by $A,B$ and then removing the 0-framed component together
with the $C$-framed one, one can easily show that $S(A,B,C)$ is
Kirby equivalent to $H(n,a,b)$.\end{proof}

Let us denote by $Q=Q(n,a,b)$ the manifold obtained from $S^3$ by
Dehn surgery along $H(n,a,b)$ or $S(A,B,C)$. Performing surgery of
$S^3$ along the 0-framed component of $S(A,B,C)$, we get
$S^2\times S^1$ such that the other three components are the
fibers of the natural fibration $S^2\times S^1 \to S^2$. It
follows that $Q$ is a Seifert manifold fibered over $S^2$ with
three exceptional fibers of types $(A,1),(B,1),(C,1)$ Here we
neglect the usual convention that the first parameters of
exceptional fibers must be positive. Our goal is to calculate
$\F(Q)$.

\begin{remark} As we have seen above, manifolds of the type $M_{H(n,a,b)}$
have a very simple structure: they are Seifert manifolds fibered
over $S^2$ with three exceptional fibers of  type $(p,\pm 1)$. The
values of $\F(M_{H(n,a,b)})$ can certainly be obtained using the
main Lescop formula (\cite{Lesc}, page 12,13), but we prefer  to
take into account the specific structure of $M_{H(n,a,b)}$ and use
simpler Lescop's formula (\cite{Lesc}, page 97).
\end{remark}

Note that the Euler number of $Q$ is
$\frac{1}{A}+\frac{1}{B}+\frac{1}{C}$ and the normalized
parameters of the exceptional fibers are $(|X|,s(X)$mod$|X|)$,
where $X=A,B,C$ and $s(X)$ is the sign of $X$. It turns out that
the signature $\sigma$ of $\L$ may be expressed via the signs of
$A,B,C$:

\begin{lemma}\label{tech}  For any numbers $A,B,C$ such that
$D=AB+AC+BC\neq 0$, $K=ABC\neq 0$, and $e=D/K>0$ we have
$\sigma=s(A)+s(B)+s(C)-1$, where $s(X)$ denotes the sign of
$X$ and $\sigma$ is the signature of the matrix
$\L=\left( \begin{array}{cc}a &n \\
n &b\end{array}\right)$ for $a=A+C, b=B+C, n=-C$.
\end{lemma}

\begin{proof}
One can extract from Lemma~\ref{Q=Q'} that the matrices $\L=\left( \begin{array}{cc}a &n \\
n &b\end{array}\right)$ and $\left( \begin{array}{cccc}
A &0&0&1 \\
0 &B&0&1 \\
0 &0&C&1 \\
1 &1&1&0
\end{array}\right)$ have the same signature. It follows that
 permutations of $A,B,C$ do not affect
the correctness of Lemma~\ref{tech}. So we may assume that $A\geq
B\geq C$.

Suppose that $K>0$. Then $D>0$ and  either $A\geq B\geq C>0$ and
$\sigma=2s(a)=2s(b)=2$ or $A>0, B<0, C<0$ and
$\sigma=2s(a)=2s(b)=-2$. In both cases we get the conclusion of the
lemma.

Now suppose that $K<0$. Since  $D<0$, we have  $\sigma=0$ and
exactly one negative number among  $A,B,C$. Therefore,
$s(A)+s(B)+s(C)=1+\sigma$.
\end{proof}

\vspace{0.5cm}

\section{Casson-Walker invariant for model manifolds}
Let $M$ be a rational homology sphere. Recall that the Lescop
invariant $\lambda_L(M)$ is related to the Casson-Walker invariant
$\lambda_w(M)$ by
$$\lambda_L(M)=\frac{1}{2}|H_1(M)|\lambda_w(M),$$
see~\cite{Lesc,Save2}. If $M$ is presented by an oriented framed
link $L$ with linking matrix $\L$, then we can rewrite that
formula as follows: $$s(D)
\lambda_L(M)=\frac{1}{2}D\lambda_w(M),$$ where $D$ is the
determinant of $\L$ and $s(D)$ is the sign of $D$. We will use
Lescop formula (\cite{Lesc}, page 97) for the Seifert manifold
$M=(S^2;(a_1,b_1),\dots, (a_m,b_m)(1,b))$ (in the original
notation $M=(Oo0| b;(a_k,b_k)_{k=1,\dots,m})$):
$$\lambda_L(M)=\left (\frac{sign(e)}{24}\left (2-m+\sum_{k=1}^m
\frac{1}{a^2_k}\right )+\frac{e|e|}{24} -\frac{e}{8}-\frac{|e|}{2} \sum_{k=1}^ms(b_k,a_k)\right)|\prod_{k=1}^{m}a_k|,$$
where $0<b_k<a_k$, $e=b+\sum_{k=1}^mb_k/a_k$ is the Euler number
of $M$, and $s(b_k,a_k)$ are the Dedekind sums.

Let us recall the definition and properties of $s(b,a)$. If $a,b$
are coprime integers, then $s(b,a)$ is defined by
$$s(b,a)=\sum_{k=1}^{|a|}\left ( \left (\frac{k}{a}\right )\right
) \left ( \left (\frac{kb}{a}\right )\right ),  $$ where

\[ \left (
\left (x \right )\right )= \left \{ \begin{array}{ll}
     x- [x]-\frac{1}{2}  & \mbox{if   $x\not \in \mathbb{Z}$}\\
0 & \mbox{if  $x\in \mathbb{Z}$}    \\
 \end{array} \right.\]
is the sawtooth function, see Fig.~\ref{saw}.

\begin{figure}[ht]
\centerline{\includegraphics[height=2.5cm]{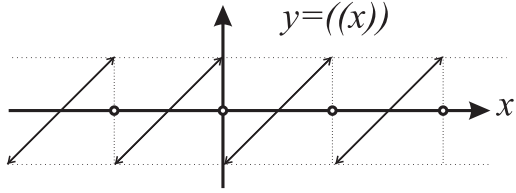}}
 \caption{The sawtooth function}
  \label{saw}
\end{figure}

 It follows from the definition that $s(b,a)$ possesses the
   property
   $s(b,a)=s(-b,-a)=-s(-b,a)=-s(b,-a)=s(b\pm a,a)$. In particular,
   $s(b,a)$ only depends on $b$ mod $a$ and $a$.

\begin{proposition}\label{lambda} Let $Q=Q(n,a,b)$ be obtained by
surgery of $S^3$ along the framed generalized Hopf link
$H(n,a,b)$. Then $\frac{1}{2}D(\lambda_w(M)- \frac{1}{4}\sigma
)=\F(H(n,a,b))$, where $\F(H(n,a,b))$ is given by
Example~\ref{compF}.
\end{proposition}

\begin{proof}
Recall from Section \ref{sectionGHL} that $Q$ is a Seifert
manifold fibered over $S^2$ with three exceptional fibers of
types $(A,1),(B,1),(C,1)$, where $A=a+n, B=b+n,C=-n$. The
Euler number of $Q$ is  $e=\frac{1}{A}+\frac{1}{B}+\frac{1}{C}$
and the normalized  parameters  of the exceptional fibers are
$(|X|,s(X) \textrm{ mod } |X|)$, where $X=A,B,C$ and $s(X)$
is the sign of $X$.

Note that reversing signs of $n,a,b $ (and hence signs of $A,B,C$
and $e$) does not affect the correctness of the conclusion of
the lemma. So we may restrict ourselves to the case $e>0$.

The proof is a cumbersome but straightforward comparison of
two expressions. Let us introduce the following notations.

\begin{enumerate}

\item $K=ABC$. Then $e=\frac{D}{K}$. Since $e>0$, $K$ and $D$ have
the same sign. We denote it by $s(D)$.

\item $P=A^2B^2+A^2C^2+B^2C^2$. Then $\frac{1}{A^2}+\frac{1}{B^2}+
\frac{1}{C^2}=\frac{P}{K^2}$.

\item
$S=AC(B-s(B))(B-2s(B))+AB(C-s(C))(C-2s(C))+BC(A-s(A))(A-2s(A))$.
In order to explain the meaning of $S$, we recall that the
Dedekind sums $s(1,\ell)$ can be calculated by the rule
$$s(1,\ell) = \frac{1}{12\ell}(\ell-s(\ell))(\ell-2s(\ell)) .$$
It follows that  $s(1,A)+s(1,B)+s(1,C)=\frac{S}{12K}$.

\item $\Sigma=A+B+C$.
\end{enumerate}

Using this notation and applying the above Lescop formula, we get

$$24\lambda_L(Q)=\varepsilon (-K+\frac{1}{K}(P+D^2-DS)-3D)$$

Simple calculation shows that $P-D^2=-2K\Sigma$ and
$2D-S=K(3(s(A)+s(B)+s(C)) -\Sigma)$. It follows that

$$24\lambda_L(Q)=
\varepsilon (-K-2\Sigma-D\Sigma+3D(s(A)+s(B)+s(C)-1)) $$ and,
since $\lambda_L(Q)=\frac{\varepsilon D}{2}\lambda_w(Q)$,

$$12D\lambda_w(Q)=
-K-2\Sigma-D\Sigma+3D(s(A)+s(B)+s(C)-1). $$

 Let us
now substitute $a=A+C, b=B+C, n=-C$ to the expression  for
$\F(H(n,a,b))$ (see Example~\ref{compF}). We get
$$24\F(H(n,a,b)) = -C^3+\Sigma C^2-\Sigma D-2\Sigma - CD$$

Let us show that $12D\lambda_w(Q)-3D\sigma-24\F(H(n,a,b))=0$.
Performing the substraction, we get
\begin{multline*}
12D \lambda_w(Q)-3D\sigma-24\F(H(n,a,b))=\\
-K+3D(s(A)+s(B)+s(C)-1-\sigma) +C^3-\Sigma C^2+CD
\end{multline*}
or, taking into account that $-K+C^3-\Sigma C^2+CD=0$,
$$12D \lambda_w(Q)-3D\sigma-24\F(H(n,a,b))= 3D(s(A)+s(B)+s(C)-1-\sigma).$$
It remains to note that $s(A)+s(B)+s(C)-1-\sigma=0$ by
Lemma~\ref{tech}.
\end{proof}

\vspace{0.5cm} {\em Proof of Main Theorem } (which states that
$\frac{1}{2}D(\lM-\frac{\sigma }{4})=\F(L)$, see
Theorem~\ref{main}). We have realized the plan of the proof
indicated in page~\pageref{plan}. Using Lemma~\ref{step1}, we
reduce the proof to the partial case of manifolds presented by
generalized framed Hopf links. Then we use
Proposition~\ref{lambda} for proving the theorem in this partial
case. $\blacksquare$

\section{Asymptotic behavior of the Casson-Walker invariant}
Let $ L=L_1\cup L_2$ be an oriented framed 2-component link. Then
$\lM$ depends on the underlying link and on the framing.
Theorem~\ref{main} allows us to understand the contribution of
those two ingredients. We use that for describing the asymptotic
behavior of $\lambda_w$ as the parameters of the framing tend to
$\infty$. For simplicity we restrict ourselves to the simplest
case when they have the form  $a=a_0t, b=b_0t$ and $t\to \infty$.

\begin{theorem}\label{linear} Let a 3-manifold $M_t$ be obtained
by surgery of $S^3$ along a framed link $L=L_1\cup L_2$ having
linking matrix $\L(t)=\left( \begin{array}{cc}a_0t &n \\
n &b_0t\end{array}\right)$. Then the following holds.

{\sc Case 1}. Suppose that $a_0+b_0\neq 0$ and $a_0b_0\neq0$. Then
$$\lambda_w(M_t)=-\frac{1}{12}(a_0+b_0)t+\frac{1}{4}\sigma+r(t),$$
where $\sigma $ is the signature of $\L(1)$ and $r(t)\to 0$ as
$t\to  \infty$.

 {\sc Case 2}. Suppose that $a_0+b_0=0 $ and $a_0b_0\neq0$. Then
$$\lambda_w(M_t)=2\frac{v_2(L_1)-v_2(L_2)}{a_0}t^{-1} +r(t),$$
where $r(t)t\to 0$ as $t\to \pm \infty$.

 {\sc Case 3}. Suppose that  $a_0b_0=0$. In order to be definite,
we assume that $b_0=0$. Then
 $$\lambda_w(M_t)=-\frac{a_0}{6n^2}(n^2-1+12v_2(L_2))t
-\frac{1}{6n^2}(n^3-n-12U(L))$$
\end{theorem}
\begin{proof} Follows easily from Theorem~\ref{main} and
Definition~\ref{formula}: we write down an explicit expression for
$\lambda_w(M_t)$ and investigate  its asymptotic behavior.
\end{proof}

Let us illustrate the behavior of $\lambda_w(M_t)$ graphically for
$a_0+b_0\neq 0$, $a_0b_0\neq0$, and $t\to \infty$. The right-hand
sides of the expression for $\lambda_w(M_t)$ (see the above proof)
and its approximation
$-\frac{1}{12}(a_0+b_0)t+\frac{1}{4}\sigma(t)$ make sense for all
(not necessarily integer) values of $t$. We show the graphs of
these functions for the generalized framed Hopf link
$H(n,a_0,b_0)$, where $n=2,a_0=3,b_0=2$ (Fig.~\ref{graph232}) and
$n=5,a_0=-3, b_0=2$ (Fig.~\ref{graph2-32}). Both graphs in
Fig.~\ref{graph232} have singularities at
$t=\frac{n^2}{\sqrt{a_0b_0}}\approx 0.8$ (because of the jump of
$\sigma(t)$).
\begin{figure}[ht]
\centerline{\includegraphics[height=7.0cm]{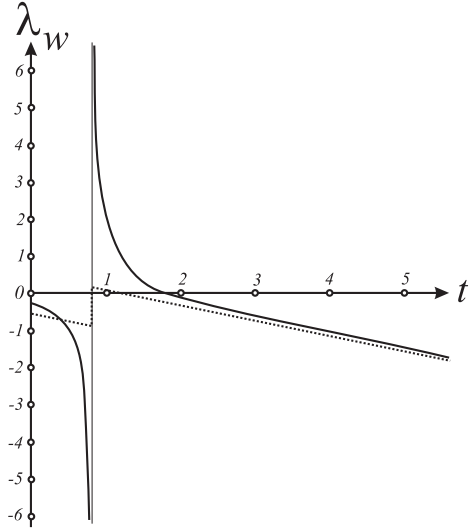}}
\caption{The behavior of  $\lambda_w(M_t)$ and its approximation
(dotted graph) for $H(2,3,2)$ }
 \label{graph232}
\end{figure}

\begin{figure}[ht]
\centerline{\includegraphics[height=7.0cm]{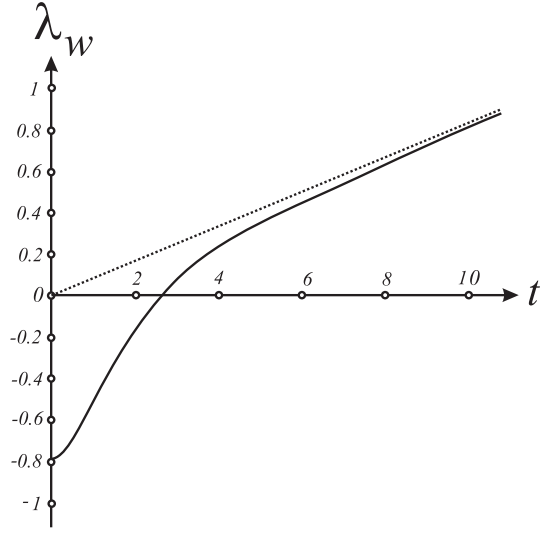}}
\caption{The behavior of  $\lambda_w(M_t)$ and its approximation
for $H(5,-3,2)$}
 \label{graph2-32}
\end{figure}

The following theorem shows the power series presentation of
$\lambda_w(M_t)$.
\begin{theorem}\label{appr} Let a 3-manifold $M(t)$ be obtained
by surgery of $S^3$ along a framed link $L=L(t)$ with linking matrix
$\L(t)=\left( \begin{array}{cc}a_0t &n \\
n &b_0t\end{array}\right)$. Suppose that $a_0b_0\neq 0$. Then for
$|t|>\frac{|n|}{\sqrt{|a_0b_0|}}$ we have
\begin{multline*}
\lambda_w(M_t)= \frac{1}{2}s(a_0b_0)-\frac{1}{12} (a_0+b_0)t+ \\
+\frac{1}{12}\sum_{k=0}^\infty \left (
(C_1-C_3(a_0+b_0))C_3^{k}t^{-(2k+1)}+C_2C_3^kt^{-(2k+2)}\right ),
\end{multline*}
 where $s(a_0b_0$ is the sign of $a_0b_0$, $C_1=\frac{(2n^2-2)(a_0+b_0)+24av_2(L_2)+24bv_2(L_1)}{ a_0b_0},$
 $C_2=\frac{-24U(L)+2n^3-2n}{a_0b_0}$, and $C_3=\frac{n^2}{a_0b_0}$.
\end{theorem}
\begin{proof} Follows form Theorem~\ref{main} and Definition~\ref{formula}.
The infinite series
in the   expression for $\lambda_w(M_t)$ arises after replacing
$D^{-1}=(a_0b_ot^2-n^2)^{-1}$ by
 $\frac{t^{-2}}{a_0b_0} \sum_{k=0}^\infty(\frac{n^2}{a_0b_0}t^{-2})^k$.
\end{proof}

\section{A skein-type relations for $U$ and $\lambda_w$}\label{sec:skein}
A usual skein relation involves diagrams of three oriented links
$L^+,L^-,L^0$. The diagrams are identical outside a small neighborhood
of one positive crossing $C$ of the diagram for $L^+$. The diagram of
the second link $L^-$ is obtained from the diagram of $L^+$ by a
crossing change at $C$ while the diagram of the oriented knot $L^0$
has no crossings at $C$.

We will consider the case when $L^+=L^+_1\cup L^+_2$,
$L^-=L^-_1\cup L^-_2$ are oriented 2-component links and $C$
is the crossing point of $L^+_1$ and $L^+_2$. Then $L^0$ is
a knot obtained by coherent fusion of  $L^+_1$ and $L_2^+$, see
Fig.~\ref{skein}. We shall refer to any triple $(L^+,L^-,L^0)$
of the above type as an {\em admissible  skein triple}.

\begin{figure}[ht]
\centerline{\includegraphics[height=3.0cm]{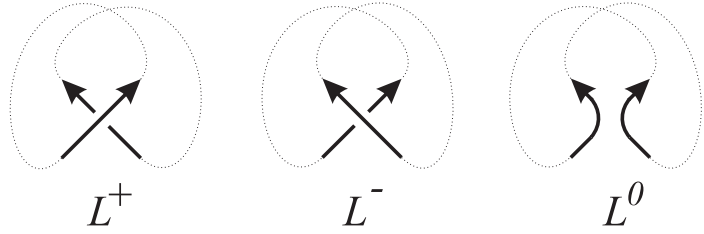}}
\caption{Three links participating in the skein-type relation}
 \label{skein}
\end{figure}

\begin{lemma}\label{skU} For any admissible skein triple
$(L^+,L^-,L^0)$ we have $U(L^+)-U(L^-)=v_2(L^0)-v_2(L_1)-v_2(L_2)$.
\end{lemma}

\begin{proof} The statement may be deduced from the identification
of $U$ with the generalized Sato-Levine invariant, see Corollary
\ref{basepoint1}. For completeness we provide a simple direct
proof using only the definition of $U$. Denote by $G^+,G^-,G^0$
Gauss diagrams corresponding to the diagrams of $L^+,L^-,L^0$.
The Gauss diagrams are almost identical. The only difference
between $G^+$ and $G^-$ is that the arrows $a^+(C), a^-(C)$
corresponding to $C$ have opposite orientations and signs.
Knot diagram $G^0$ is obtained by coherent fusion of the circles
of $G^+$  along $a^+(C)$. Chose a base point in the first circle
of $G^+$ just before the initial point of $a^+(C)$. We may assume
that there are no endpoints of other arrows on small arcs containing
the endpoints of $a^+(C),a^-(C)$, and on arcs of $L^0$ obtained by
their fusion. See Fig.~\ref{fuse}, where those free-of-endpoints
arcs are shown dotted and the complementary arcs are numbered by
1,2. For reader's convenience we have also placed arrows diagrams
for $U$ and $v_2$.
\begin{figure}[ht]
\centerline{\includegraphics[height=3.6cm]{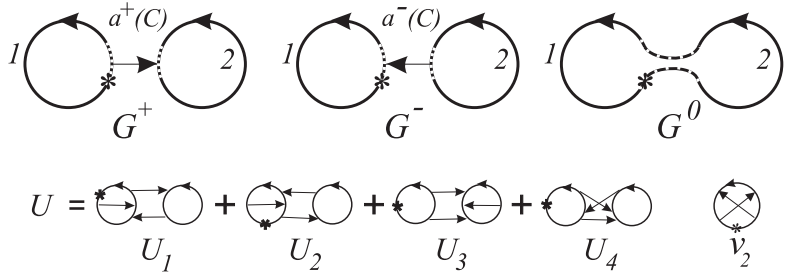}}
\caption{Skein triple of Gauss diagrams}
 \label{fuse}
\end{figure}

Let us analyze representations of $U_k,1\leq k\leq 4,$  in
$G^\pm$. As in the proof of Lemma~\ref{step1}, we call a
representation $\varphi \colon U_k\to G^\pm$  significant, if its
image contains $a^\pm(C)$. Insignificant representations  do not
contribute to the difference $U(L^+)-U(L^-)$.

 Let $\varphi$ be a significant representation of
$U_k, 1\leq k\leq 4,$ in $G^{\pm }$ and let $a_k\subset U_k$ be
the first arrow we meet travelling from the base point along the
first circle of $U_k$. Then the careful choice of the base point
for $G^{\pm}$ tells us that $\varphi$ takes $a_k$ to $a^\pm (C)$.
It follows that there is no significant representation of $U_1$ in
$G^{\pm}$ (since the endpoints of $a_1$ lie in the same circle of
$U_1$ while the endpoints of $a^\pm (C)$ lie in different circles
of $G^{\pm})$.

 By similar reason there are no significant
representations of $U_k, 2\leq k\leq 4$ in $G^-$. Indeed, $a_k$ is
directed from the first circle of $U_k$ to the second one while
$a^-(C)$ is directed from the second circle of $G^-$ to the first
one.

 Suppose that $\varphi$ is a significant
representation of $U_k, 2\leq k\leq 4$ in $G^+$.  Let us
coherently fuse the circles of $U_k$ along $a_k$. It is easy to
see that we get an arrow diagram $A_4$ (shown in Fig.~\ref{arrex}
and Fig.~\ref{fuse}) for calculation of $v_2$, together with the
corresponding representation $\varphi'\colon A_4\to G^0$. The
values of $\varphi$ and $\varphi'$ are equal. Vice versa, any
representation $A_4\to G^0$ such that arrows do not all land on
the image (under fusion) of the same circle of $G^+$, determines a
representation $U_k\to G^+$ having the same value. The only
representations $A_4\to G^0$ which have no corresponding
representation $U_k\to G^+$ are actually
 representations of $A_4$ in the Gauss diagram either of
 the first or the second circle of $G^+$. It follows that
$U(L^+)-U(L^-)=v_2(L^0)-v_2(L_1)-v_2(L_2)$.
\end{proof}

As a result we get the following behavior of $\lambda_w$ under a
crossing change involving two components:
\begin{theorem}\label{altlambda} For any admissible skein triple
$(L^+,L^-,L^0)$ we have
$$\F(L^+)-\F(L^-)=v_2(L_1)+v_2(L_2)-v_2(L^0)+\frac{1}{4}(n^2-n)+
\frac{1}{12}(a+b)(2n-1)$$ thus
\begin{multline*}
\frac{D^+}{2}\lambda_w(L^+)
-\frac{D^-}{2}\lambda_w(L^-)= \frac{D^+}{8}\sigma^+ -\frac{D^-}{8}\sigma^- +\\
+v_2(L_1)+v_2(L_2)-v_2(L^0)+\frac{1}{4}(n^2-n)+\frac{1}{12}(a+b)(2n-1)
\end{multline*}
\end{theorem}
\begin{proof}
Directly follows from Lemma~\ref{skU} and Theorem \ref{main}.
\end{proof}

\section{An Alternative Formula}

Consider the following linear combination
$U'=U_1+U_2+\frac{1}{2}(U_3+U_3')+\frac{1}{2}(U_4-U_4')$
of arrow diagrams, see Fig.~\ref{U'}.

\begin{figure}[ht]
\centerline{\includegraphics[height=2.7cm]{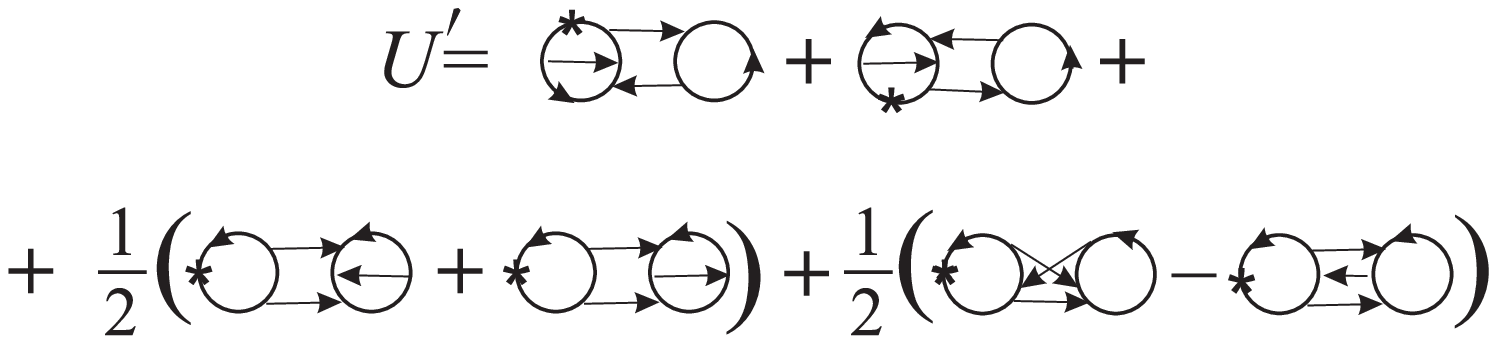}}
\caption{Another remarkable linear combination of arrow diagrams}
 \label{U'}
\end{figure}

\begin{lemma}\label{uandu'}
Let $L=L_1\cup L_2$ be an oriented 2-component link. Denote by
$L'$ the link $L_1\cup L'_2$ obtained from $L$ by reversing the
orientation of the second component. Let $G$ and $G'$ be their
based Gauss diagrams (assuming that the base points are in the
first components). Then $\left<U_1,G\right>=\left<U_1,G'\right>$,
$\left<U_2,G\right>=\left<U_2,G'\right>$, $\left<U_3,G\right>=
\left<U'_3,G'\right>$, $\left<U_4,G\right>=-\left<U'_4,G'\right>$,
and $\left<U',G\right>=\frac{1}{2}\left(\left<U,G\right>+
\left<U,G'\right>\right )$.
\end{lemma}
\begin{proof} Note that $G$ and $G'$ actually coincide. The only
difference is that their second circles have opposite orientations
and that arrows joining different circles have opposite signs.
First two equalities of the conclusion of the lemma are evident,
since any representation of $U_i, i=1,2,$ to $G$ determines a
representation of $U_i$ to $G'$, and vice-versa. The images of those
representations contains the same arrows. Since exactly two of these
arrows join differen components, the values of the representations
are equal. Similarly, any representation of $U_j,j=3,4$ determines
a representation of $U'_j$ to $G'$, and vice-versa. The values of
those representations are the same for $j=3$ and have opposite signs
for $j=4$. This is because  the number of arrows in the images of
representations is 2 for $j=3$ and 3 for $j=4$. Taking the sums,
we get $\left<U',G\right>=\frac{1}{2}\left(\left<U,G\right>+
\left<U,G'\right >\right)$.
\end{proof}

The following proposition and theorem are similar to
Proposition~\ref{prop1} and Theorem~\ref{main}.

\begin{proposition}\label{prop2}  $U'$ determines an invariant
$U'(L)=\left<U',G\right>$ for non-oriented links of two numbered
components.
\end{proposition}

\begin{proof} It follows from Lemma~\ref{uandu'} that
$\left <U',G\right > =\frac{1}{2}\left (\left <U ,G\right >+\left
<U,G'\right >\right )$.  Therefore, $U'(L)=\left <U',G\right > $ is
an invariant of oriented links by Proposition~\ref{prop1}. On the
other hand,  $\left <U',G\right >$ is invariant under reversing
orientation of one of its components, since  the above expression
for it is symmetric.
\end{proof}

\begin{theorem}\label{main'} For any non-oriented framed
2-component link $L=L_1\cup L_2$ we have
$\displaystyle{\frac{1}{2}D(\lM-\frac{1}{4}\sigma )=\F'(L)}$,
where $$\F'(L)=av_2(L_2)+bv_2(L_1)-U'(L)+\frac{1}{24}(a+b)(2n^2-ab
-2)$$
\end{theorem}

\begin{proof}  Let us orient $L$ and denote by $L'$ the oriented
framed link obtained from $L$ by reversing orientation of $L_2$.
Note that the linking number  $n$ of the components of $L$ and the
linking number $n'$ of the components of $L'$ have the same
modules and different signs, that is, $n'=-n$. All other variables
in the expressions for $\F(L)$ and $\F(L')$ (see
Definition~\ref{formula}) except of $U(L),U(L')$ are the same. In
other words, $a,b,D,\sigma$ and both $v_2$ for $L$ coincide with
the corresponding variables for $L'$. It follows from
Lemma~\ref{uandu'} that $\F'(L)=\frac{1}{2}(\F(L)+\F(L')$. Since
$\F(L)=\frac{1}{2}D(\lM-\frac{1}{4}\sigma )=
\frac{1}{2}D(\lambda_w (M_{L'})-\frac{1}{4}\sigma )=\F(L')$ by
Theorem~\ref{main}, we get the conclusion.
\end{proof}

\section{Results of computer experiments}
The formula for $\lambda_w(M)$ from Theorem~\ref{main} is very
convenient for calculation. If a diagram of a link framed by
integers of reasonable size has about a dozen crossing points,
then the manual calculation takes a few minutes. A simple computer
program written by V. Tarkaev accepts Gauss codes and takes only
seconds for calculating $\lambda_w$ for framed links with
thousands crossings. We present here a few results of calculation
$\lambda_w$ for all rational homology 3-spheres which can be
presented by diagrams of 2-component (but not of 1-component)
links with $\leq 9$ crossings and black-board framings. We thank
S. Lins, who kindly prepared for us a list of Gauss codes of those
links.

\begin{itemize}
\item Number of different manifolds:  194 \item Number of
different values of $|\lambda_w|$ for these manifolds: 66 \item
Numbers of different manifolds having given values of
$|\lambda_w|$ are presented in Table 1.
\end{itemize}
 \begin{table} [ht]\label{tab1}
  \begin{center}
 \begin{tabular}{|c c|c|c c|c|c c|c|c c|} \cline{1-2}  \cline{4-5}
 \cline{7-8} \cline{10-11}
$\mathbf{|\lambda_w|}$&$\mathbf{\# M^3}$&
&$\mathbf{|\lambda_w|}$&$\mathbf{\# M^3}$&
&$\mathbf{|\lambda_w|}$&$\mathbf{\# M^3}$&
&$\mathbf{|\lambda_w|}$&$\mathbf{\# M^3}$
              \\ \cline{1-2}  \cline{4-5}\cline{7-8}\cline{10-11}
$0$&$19$& &$9/64$&$1$&    &$11/36$&$4$& &$25/44$&$1$ \\ \cline{1-2}
\cline{4-5}\cline{7-8}\cline{10-11}
$1/64$&$1$& &$4/27$&$2$&    &$5/16$&$7$& &$16/27$&$1$ \\ \cline{1-2}
\cline{4-5}\cline{7-8}\cline{10-11}
$1/36$&$4$& &$5/32$&$1$&     &$9/28$&$2$& &$3/5$&$1$ \\
\cline{1-2} \cline{4-5}\cline{7-8}\cline{10-11}
 $1/32$&$1$& &$11/64$&$1$&      &$11/32$&$2$& &$23/36$&$6$ \\
\cline{1-2} \cline{4-5}\cline{7-8}\cline{10-11}
$1/28$&$2$& &$5/28$&$1$&    &$9/26$&$1$& &$11/16$&$3$ \\
\cline{1-2} \cline{4-5}\cline{7-8}\cline{10-11}
$1/26$&$1$& &$3/16$&$10$&    &$13/36$&$2$& &$25/36$&$3$ \\
\cline{1-2} \cline{4-5}\cline{7-8}\cline{10-11}
$3/64$&$1$& &$5/26$&$1$&    &$10/27$&$2$& &$23/32$&$1$ \\
\cline{1-2} \cline{4-5}\cline{7-8}\cline{10-11}
$1/16$&$8$& &$7/36$&$1$&    &$3/8$&$5$& &$29/36$&$1$ \\
\cline{1-2} \cline{4-5}\cline{7-8}\cline{10-11}
$3/44$&$1$& &$1/5$&$4$&    &$25/64$&$1$& &$15/16$&$3$ \\
\cline{1-2} \cline{4-5}\cline{7-8}\cline{10-11}
$2/27$&$2$& &$13/64$&$1$&    &$2/5$&$1$& &$1$&$7$ \\ \cline{1-2}
\cline{4-5}\cline{7-8}\cline{10-11}
$5/64$&$1$& &$9/44$&$1$&    &$13/32$&$1$& &$19/16$&$1$ \\
\cline{1-2} \cline{4-5}\cline{7-8}\cline{10-11}
$3/32$&$2$& &$2/9$&$5$&    &$19/44$&$1$& &$3/2$&$1$ \\ \cline{1-2}
\cline{4-5}\cline{7-8}\cline{10-11}
$1/10$&$2$& &$15/64$&$1$&    &$7/16$&$5$& &$2$&$16$ \\
\cline{1-2} \cline{4-5}\cline{7-8}\cline{10-11}
$5/44$&$2$& &$1/4$&$5$&    &$4/9$&$2$& &$3$&$2$ \\ \cline{1-2}
\cline{4-5}\cline{7-8}\cline{10-11}
$3/26$&$1$& &$13/44$&$2$&    &$13/28$&$1$& &$4$&$6$ \\
\cline{1-2} \cline{4-5}\cline{7-8}\cline{10-11}
$1/8$&$4$& &$8/27$&$2$&    &$17/36$&$4$& &$ -$&$ -$ \\ \cline{1-2}
\cline{4-5}\cline{7-8}\cline{10-11}
$5/36$&$2$& &$3/10$&$1$&    &$1/2$&$6$& &$- $&$- $ \\ \cline{1-2}
\cline{4-5}\cline{7-8}\cline{10-11}

 \end{tabular}
 \caption{ How many manifolds have a given value of $|\lambda_w|$ }
 \end{center}
 \end{table}

We see from the table that for the set of 194 manifolds under
consideration the most popular values of $|\lambda_w|$ are $0$ (19
manifolds), $2$ (16 manifolds), $3/16$ (10 manifolds), $1/16$ (8
manifolds), $5/16$ and $1$ (7 manifolds each). Exactly 30 values
of $|\lambda_w|$ are taken by only one manifold each. In other
words, those manifolds are determined by $|\lambda_w|$. In
average, each value is taken by only 3 different manifolds from
the list. This shows that the Casson-Walker invariant is
unexpectedly informative. For example, the number  of different
first homology groups of the manifolds under consideration is 17,
so the average number of manifolds having a given group is about
11.4.


\begin{thebibliography}{10000}

\bibitem{AkMc}
{\bf Akbulut, S., McCarthy, J. D. } {\em  Casson's invariant for
oriented homology $3$-spheres. An exposition.} ~// Mathematical
Notes, 36. Princeton University Press, Princeton, NJ, (1990).
xviii+182 pp.

\bibitem{AMR} {\bf Akhmetiev, P. M., Male\v{s}i\v{c}, J.,
Repov\v{s}, D. }{\em A formula for the generalized Sato--Levine
invariant.} ~// Mat. Sbornik 192:1 (2001), 3--12; English transl.
Sb. Math. (2001) 1--10.

\bibitem{AR} {\bf Akhmetiev, P. M., Repov\v{s}, D. }
{\em A generalization of the Sato--Levine invariant.} ~// Trudy
Mat. Inst. im. Steklova 221 (1998), 69--80; English transl. Proc.
Steklov Inst. Math. 221 (1998), 60--70.

\bibitem{host} {\bf Hoste, J.} {\em A formula for Casson's invariant.}
~// Trans. A.M.S. 297 (1986), 547--562.

\bibitem{GGP} {\bf Garoufalidis, S., Goussarov, M., Polyak, M. }
{\em  Calculus of clovers and finite type invariants of 3-manifolds.}
~// Geometry and Topology 5 (2001), 75--108.

\bibitem{GuPoVi}{\bf Goussarov, M., Polyak  M., Viro  O. } {\em  Finite-type
invariants of classical and virtual knots.} ~//  Topology 39 (2000),
no. 5, 1045--1068.

\bibitem{Joha1} {\bf Johannes, J.} {\em  A Type 2 polynomial invariant of links
derived from the Casson-Walker invariant.} ~// J. Knot Theory
Ramifications 8 (1999), 491--504.

\bibitem{Joha2}  {\bf Johannes, J.} {\em  The Casson-Walker-Lescop invariant and
link invariants} ~// J. Knot Theory Ramifications 14:4 (2005), 425--433.

\bibitem{KiLi} {\bf Kirk, R., Livingston, C.} {\em Vassiliev invariants of two
component links and the Casson-Walker invariant} ~// Topology 36
(1997), 1333--1353.

\bibitem{Lesc} {\bf Lescop, C.} {\em  Global surgery formula for Casson-Walker
invariant} ~//  Annals of Mathematics Studies 140, Princenton
University Press, Princenton, (1996).

\bibitem{Liv} {\bf Livingston, C.}
{\em Enhanced linking numbers} ~// Amer. Math. Monthly 110 (2003)
361--385.

\bibitem{Mat} {\bf  Matveev, S.} {\em  Generalized  surgery of 3-manifolds
and representations of homology spheres.} ~// Mat. Zametki 42
(1987), 268-278 (English translation in: Math. Notices Acad. Sci.
USSR 42:2 (1987), 651--656).

\bibitem{Mel} {\bf Melikhov, S.}
{\em Colored finite type invariants and a multi-variable analogue of the
Conway polynomial} ~// preprint {\tt math.GT/0312007} (2003).

\bibitem{NO} {\bf Nakanishi, Y., Ohyama, Y.}
{\em Delta link homotopy for two component links, III} ~//
J. Math. Soc. Japan 55:3, (2003), 641--654

\bibitem{Polyak} {\bf Polyak, M. } {\em  On the algebra of arrow diagrams.} ~//
Lett. Math. Phys. 51:4, (2000), 275--291.

\bibitem{PoVi1} {\bf  Polyak, M., Viro, O.} {\em   On the Casson knot
 invariant.} ~//  Knots in Hellas '98, Vol. 3 (Delphi).
J. Knot Theory Ramifications 10 (2001), no. 5, 711--738.

\bibitem{PoVi2}{\bf Polyak, M., Viro, O. } {\em   Gauss diagram
formulas for Vassiliev invariants. } ~// Internat. Math. Res.
Notices (1994), no. 11, 445ff., approx. 8 pp. (electronic).

\bibitem{Save1}{\bf Saveliev, N.} {\em Lectures on the Topology of
3-Manifolds. An introduction to the Casson invariant }~// Walter de
Gruyter, (1999).

\bibitem{Save2}{\bf Saveliev, N.} {\em Invariants for Homology
3-spheres}   ~// Springer (2003).

\bibitem{rolf} {\bf Rolfsen, D.} {\em   Knots and Links} ~// Mathematics Lecture
Series 7, Publish of Perish, Houston (1976).

\bibitem{Walk} {\bf Walker, K.} {\em   An extension of Casson's invariant,} ~//  Annals
of Mathematics Studies 126, Princeton University Press, Princeton
(1992).




\end{thebibliography}
\end{document}